\documentclass[11pt]{amsart}
\usepackage{latexsym,amsfonts,amssymb,amsthm,amsmath,mathrsfs,color,amscd,graphicx,fullpage,fancyhdr,parskip,hyperref}

\newcommand{\s}[1]{{\mathcal #1}}
\newcommand{\sr}[1]{{\mathscr #1}}
\newcommand{\bb}[1]{{\mathbb #1}}

\newtheorem{theorem}{Theorem} 
\newtheorem{corollary}[theorem]{Corollary}

\newtheorem{lemma}[theorem]{Lemma}
\newtheorem{proposition}[theorem]{Proposition}

\newtheorem{definition}[theorem]{Definition}

\newtheorem{assumption}[theorem]{Assumption}

\numberwithin{equation}{section}
\numberwithin{theorem}{section}

\begin{document}

\title{Linear Quadratic Mean Field Type Control and Mean Field Games with Common Noise, with Application to Production of an Exhaustible Resource}
\author{P. Jameson Graber}
\thanks{The author is grateful to be supported in this work by the National Science Foundation under NSF Grant DMS-1303775.}
\address{International Center for Decision and Risk Analysis\\
Naveen Jindal School of Management\\
The University of Texas at Dallas\\
800 West Campbell Rd, SM30\\
Richardson, TX 75080-3021\\
Phone: (972) 883-6249}
\email{pjg140130@utdallas.edu}


\dedicatory{Version: \today}

\maketitle

\begin{abstract}
We study a general linear quadratic mean field type control problem and connect it to mean field games of a similar type.
The solution is given both in terms of a forward/backward system of stochastic differential equations and by a pair of Riccati equations.
In certain cases, the solution to the mean field type control is also the equilibrium strategy for a class of mean field games.
We use this fact to study an economic model of production of exhaustible resources.

\emph{Keywords: mean field type control, mean field games, linear-quadratic, optimal control, riccati equations, exhaustible resource production}

MSC: 93E20
\end{abstract}

\section{Introduction} \label{sec:intro}

The purpose of this work is to study mean field games and mean field type control problems of linear quadratic type, primarily those motivated by a certain kind of application to economics, the quintessential example being the production of an exhaustible resource.
Let us recall that mean field game theory was introduced by the parallel works of Caines, Huang and Malham\'e \cite{huang2006large} and of Lasry and Lions \cite{lasry06,lasry06a,lasry07}, with a general aim to study the interactions of large populations of rational actors.
A \textit{mean field game} refers essentially to an \textit{equilibrium} which occurs when the strategy employed by a representative agent of a given crowd is optimal given the costs imposed by that crowd.
A useful overview of the topic can be found in the notes of Cardaliaguet \cite{cardaliaguet2010notes} based on the lectures of Lions at the Coll\`ege de France \cite{lions07}.
For an introduction to both theory and applications of mean field games, see especially the Paris-Princeton Lectures of Gu\'eant, Lasry and Lions \cite{gueant2011mean}.
See also the survey by Gomes \cite{gomesmean}.
For a probabilistic analysis of mean field games, see Carmona and Delarue \cite{carmona2013probabilistic}.
We also mention that numerical methods have been important in the development of mean field game theory; see especially Achdou and Cappuzzo-Dolcetta \cite{achdou2010mean} and Achdou, et al.~\cite{achdou2012mean}.

A related but distinct concept is that of \textit{mean field type control}.
In this case, the goal is to assign a strategy to all agents at once, such that the resulting crowd behavior is optimal with respect to costs imposed on a \textit{central planner}.
For a comparison of mean field games and mean field type control, see the book of Bensoussan, Frehse, and Yam \cite{bensoussan2013mean} as well as the article by Carmona, Delarue, and Lachapelle \cite{carmona2013control}.
A key reference is the work of Carmona and Delarue \cite{carmona2015forward}, which characterizes solutions to the mean field type control problem in terms of a stochastic maximum principle for McKean-Vlasov type dynamics.

While mean field type control is conceptually distinct from mean field games, and although in general an optimal control on the one hand is not an equilibrium strategy on the other, nevertheless in many cases a mean field Nash equilibrium is also the solution to an optimal control problem, as pointed out at least as early as \cite{lasry07}.
Many researchers have used this insight to generate results concerning existence and uniqueness \cite{cardaliaguet2013weak,cardaliaguet2014mean,cardaliaguet2015second,meszaros2015variational,cardaliaguet2015first} as well as computation of mean field Nash equilibria \cite{benamou2016variational}.
It must be understood that the overall \textit{minimized} cost is smaller than the total cost to all individual players; the difference between the two is called the \textit{price of anarchy} (see the discussion in \cite{benamou2016variational}).
The present work also highlights this point of view.
Our motivating example from economics, while conceptually construed as a Nash equilibrium, can be solved via a reformulation as an optimal control problem.
For this reason our results are more heavily inclined toward the study of mean field type control, even though, \textit{a priori,} we are interested in mean field games.

One of the most natural ways to apply mean field game theory is to such fields as economics and systemic risk, since here the critical questions concern the behavior of large numbers of individuals motivated by similar incentives.
See, for instance, the thesis of Gu\'eant \cite{gueant2009mean} on mean field games and economics, as well as related work by Gu\'eant et al.~\cite{gueant2008application} and by Lachapelle, Salomon, and Turinici \cite{lachapelle2010computation}; the influential paper of Lucas and Moll \cite{lucas11}; the work of Carmona, Fouque and Sun on systemic risk \cite{carmona13}; and many other references, many of which can be found in the survey articles \cite{achdou2014partial,burger2014partial,gomes2014socio}.

In this paper we are particularly motivated by a model of the production of an exhaustible resource, such as oil.
We draw our inspiration from a model found in \cite{gueant2011mean} and later adapted by Chan and Sircar in \cite{chan2014bertrand,chan2015fracking}.
(See also the work of Bauso, Tembine, and Basar \cite{bauso2012robust}.)
A well-posedness result for a related system of partial differential equations appears in a paper by Bensoussan and the present author \cite{graber2015existence}.
The basic structure of the model is as follows.
Let $X$ represent the amount of remaining reserves held by a firm, $v$ the level of production, and the dynamics governed by a linear stochastic differential equation:
$$
dX(s) = -v(s)ds + \nu X(s) dW(s) + \nu_0X(s) dW_0(s), \ X(t) = x.
$$
A (mean field) Nash equilibrium is obtained whenever each firm has solved the profit maximization problem
$$
\sup_{v} \bb{E}\left[\int_t^T e^{-\mu(s-t)}v(s)(2\alpha - 2\beta \bar \psi(s) - v(s))ds - e^{-\mu(T-t)}|X(T)|^2 \right]
$$
where $\mu$ is the discount rate and $\psi(s) = \bb{E}[\hat v(s)|\sr{F}_s^0]$ is the conditional expectation (given the common noise) of the equilibrium strategy $\hat v(s)$.
The reason for the appearance of the conditional expectation of the equilibrium strategy in the objective functional is that the market price is determined by taking an average.
In this model, we have simplified the calculation of the price considerably by taking a \textit{linear} demand schedule (cf.~Chan and Sircar \cite{chan2014bertrand}).
Hence the mean field game is of linear-quadratic form.

Linear quadratic models were among the first to receive full mathematical treatment by researchers studying mean field games.
For the infinite time horizon case, we mention the work of Gu\'eant \cite{gueant2009reference}, of Caines, Huang, and Malham\'e  \cite{huang2003individual,huang2007large,huang2010social,huang2012social}, and of Li and Zhang \cite{li2008asymptotically}.
For the finite time horizon case, we mention in particular the work of Bensoussan et al.~\cite{bensoussan2011linear}, which deals both with mean field games and mean field type control problems.
See also \cite{bardi2012explicit,gueant2012mean,lachapelle2010human,lachapelle2010computation,lachapelle2011on,tembine2011risk,yang2011mean}.
For the discrete time case, see the recent work of Ni, Zhang, and Li \cite{ni2015indefinite}.

On the mean field type control side, the general linear quadratic case without common noise has been dealt with in the work of Yong \cite{yong2013linear}.
Unlike most other references, his result allows the cost functional to depend on the \textit{expected value of the control} as well as of the state variable.
In that work the motivation for this dependence comes from problems involving the minimization of the variance, both of the control and state variables.
In this paper, by contrast, we are interested in economic applications, and in fact we will see that the mean field \textit{game} described above can be written as a mean field type \textit{control} problem, in that the equilibrium strategy is the optimal decision for a central planner trying to minimize the following objective functional:
$$
\bb{E}\left[\int_t^T e^{-\mu(s-t)}v(s)(v(s) + \beta \bb{E}[v(s)|\sr{F}_s^0] - 2\alpha)ds + e^{-\mu(T-t)}|X(T)|^2 \right].
$$

Three other recent works of particular importance in the context of mean field type control have recently been published by Pham and Wei \cite{pham2015bellman,pham2016dynamic} and by Pham \cite{pham2016linear}.
Pham and Wei develop a dynamic programming technique (see also \cite{lauriere2014dynamic}), with corresponding Hamilton-Jacobi equations on an infinite dimensional space of probability measures, for solving mean field type optimal control problems, first without \cite{pham2015bellman} and then with common noise \cite{pham2016dynamic}.
Both of these citations include brief applications to linear-quadratic problems.
We note that \cite{pham2015bellman} reproduces and slightly extends the results of \cite{yong2013linear} using dynamic programming.
Meanwhile \cite{pham2016dynamic} treats the case of a common noise; however, in that case the control is already adapted to the common noise, which allows the authors to prove that the distribution of the state variable is Markovian (a crucial step in proving the dynamic programming principle).
In \cite{pham2016linear} Pham further develops the theory of linear quadratic problems with common noise by considering \textit{random coefficients}; again in this reference the control is adapted to the common noise.
The present work is complementary to  \cite{pham2015bellman,pham2016dynamic,pham2016linear} in that (a) we investigate the case where the control is not necessarily adapted to the common noise, and both the control and its conditional expectation (given the common noise) are variables in the dynamics and the cost; (b) we explicitly consider the connection with mean field games; and (c) we are inspired by a particular application to economics, in contrast to the financial applications of these other works.
However, we only consider deterministic coefficients.

The distinguishing features of the present work are as follows.
First, we will deal with the case of a \textit{common noise,} which is taken to represent an inherent uncertainty in nature affecting simultaneously all the agents participating in the game (or being controlled by a central planner).
Second, as in \cite{yong2013linear,pham2015bellman} but unlike most of the previous references, we will consider the \textit{conditional expectation of the control variable} (or the equilibrium strategy in the case of mean field games) as a factor in the dynamics and quadratic cost.
Finally, we use an economic model from recent literature to illustrate the applicability of our general framework.

Mean field games with common noise have been analyzed recently using several different approaches.
First we mention the master equation, first introduced by Lions in his lectures at the \textit{Coll\`ege de France} \cite{lions07} (see also \cite{cardaliaguet2010notes,carmona2014master,bensoussan2015interpretation}).
This is a partial differential equation on the Wasserstein space of probability measures, where the common noise is encoded in a second-order derivative with respect to the measure variable.
A well-posedness result for the master equation can be found in \cite{cardaliaguet2015master}, including in the case of a common noise.
A different approach, from a probabilistic point of view,
can be found in the work of Ahuja \cite{ahuja2016wellposedness} and Carmona, Delarue, and Lacker \cite{carmona2014mean}.
The former starts with the stochastic maximum principle for a representative player and uses a fixed point argument to find a mean field Nash equilibrium.
The latter also develops a theory of weak solutions for which there is a quite general existence result.

A question of particular interest in the case of common noise is whether the mean field equilibrium can act as an approximate solution of an $N$-player game for large $N$.
Such an approximation is known as ``$\epsilon$-Nash equilibrium" \cite{cardaliaguet2010notes,huang2003individual,huang2007large,huang2006large}.
The first largely comprehensive account of this problem for mean field games with common noise was given by Lacker in \cite{lacker2014general}.
Note, however, that this general result does not cover all linear quadratic models.
In the present work we will prove that mean field game solutions serve as $\epsilon$-Nash equilibria for large $N$-player games, but only in a special case which most directly motivates our results, namely when the \textit{game} is in fact solved by an \textit{optimal control problem.}


To complete this introduction, we give an outline of the rest of the paper.
In Section \ref{sec:mftc} we completely solve the linear-quadratic mean field type control problem.
We characterize the solution both in terms of a stochastic maximum principle (forward-backward system of stochastic differential equations) and Riccati equations.
In Section \ref{sec:mfg} we discuss the question of Nash equilibrium.
Rather than seek the most general possible solution, our main goal will be to provide criteria that allow an equilibrium to be interpreted as a global optimal solution for a mean field control problem, in which case the results of Section \ref{sec:mftc} can be applied to show that there is a unique solution to the mean field game.
Finally, in Section \ref{sec:econ}, we describe and solve a linear-quadratic version of an economic production model.

\section{The Mean Field Type Control Problem}
\label{sec:mftc}
Fix an initial time $t$ and a final time $T$.
Let $(\Omega,(\sr{F}_s)_{s \in [t,T]},\bb{P})$ be a complete filtered probability space.
We suppose that $W(s),W_0(s)$ are independent $(\sr{F}_s)_{s \in [t,T]}$-Wiener processes, and that $x$ is a random variable independent of $W(s),W_0(s)$.
(Here $W_0(s)$ is considered as the \textit{common noise.})
Throughout we will denote by $(\sr{F}_s^0)_{s \in [t,T]}$ the filtration generated by $W_0(s), s \in [t,T]$.
If $X(s)$ is any stochastic process adapted to $(\sr{F}_s)_{s \in [t,T]}$, we denote $\bar X(s) = \bb{E}[X(s)|\sr{F}_s^0]$, the conditional expectation of $X(s)$ given $W_0(s)$.

The linear-quadratic mean field type control problem is formulated as follows.
An \textit{admissible control} is defined to be a square integrable $(\sr{F}_s)_{s \in [t,T]}$-adapted process with values in $\bb{R}^m$.
The corresponding state variable $X(s)$ is an $\bb{R}^d$-valued adapted process satisfying the dynamics
\begin{multline} \label{lq dynamics}
dX(s) = \left\{A(s)X(s) + \bar{A}(s)\bar X(s) + B(s)v(s) + \bar{B}(s)\bar v(s) \right\}ds
\\
+ \left\{C(s)X(s) + \bar{C}(s)\bar X(s) + D(s)v(s) + \bar{D}(s)\bar v(s)\right\}dW(s)
\\
+ \left\{F(s)X(s) + \bar{F}(s)\bar X(s) + G(s)v(s) + \bar{G}(s)\bar v(s)\right\}dW_0(s), \ \ X(t) = x.
\end{multline}
Let $\langle .,. \rangle$ be the inner product on Euclidean space.
The objective functional is given by
\begin{multline} \label{mftc objective functional}
J^{LQ}_{x,t}(v) = \bb{E}\biggl\{\int_t^T \bigl[\langle Q(s)X(s),X(s) \rangle  + \langle \bar Q(s)\bar X(s),\bar X(s) \rangle  
+ \langle R(s)v(s),v(s) \rangle + \langle \bar R(s) \bar v(s),\bar v(s) \rangle
\\
+ 2\langle S(s)X(s),v(s) \rangle + 2\langle \bar S(s) \bar X(s),\bar v(s) \rangle
+ 2\langle q(s),X(s)\rangle + 2\langle \bar q(s),\bar X(s) \rangle 
\\
+ 2\langle r(s),v(s) \rangle + 2\langle \bar r(s), \bar v(s) \rangle \bigr]ds
+ \langle HX(T),X(T) \rangle + \langle \bar H \bar X(T),\bar X(T) \rangle \biggr\}.
\end{multline}
We seek an optimal control $\hat v$ such that
\begin{equation}
\label{LQ optimal control problem}
J^{LQ}_{x,t}(\hat v) = \inf_{v} J^{LQ}_{x,t}(v).
\end{equation}
Let us now give some standing assumptions on the coefficients.
First, we define $\s{S}^n$ to be the set of all $n \times n$ symmetric matrices with real entries.
Now we state the following:
\begin{assumption} \label{coeffs}
The coefficient matrices satisfy
\begin{enumerate}
\item $A,\bar A,C, \bar C, F, \bar F \in L^\infty(0,T;\bb{R}^{d \times d})$
\item $B, \bar B,D, \bar D, G, \bar G \in L^\infty(0,T;\bb{R}^{d \times m})$
\item $Q, \bar Q \in L^\infty(0,T;\s{S}^d), \ R,\bar R \in L^\infty(0,T;\s{S}^m), \ H,\bar H \in \s{S}^d$
\item $H, H + \bar H \geq 0$, and for some $\delta_1 \geq 0, \delta_2 > 0$, $Q, Q + \bar{Q} \geq \delta_1 I$ and $R, R + \bar R \geq \delta_2 I$
\item $S,\bar S \in L^\infty(0,T;\bb{R}^{m \times d}); q, \bar q \in L^\infty(0,T;\bb{R}^d); r, \bar r \in L^\infty(0,T;\bb{R}^m)$
\item $\|S\|_\infty^2, \|S+\bar S\|_\infty^2 < \delta_1 \delta_2$ if $\delta_1 > 0$, $S = \bar S = 0$ otherwise.
\end{enumerate}
\end{assumption}
Under these assumptions, the dynamics \eqref{lq dynamics} are well-posed in the following sense:
\begin{lemma}
Let $v$ be an admissible control process and $x$ an $L^2(\Omega,\sr{F}_t,\bb{P})$ random variable. Then there exists a unique $(\sr{F}_s)_{t \leq s \leq T}$-adapted state process $X(s)$ satisfying \eqref{lq dynamics} with a continuous version such that
$$
\bb{E} \int_t^T |X(s)|^2 ds < \infty.
$$
\end{lemma} 

\begin{proof}
This is proved in a straightforward manner using a fixed point argument, following standard theory for McKean-Vlasov dynamics.
(See \cite{yong2013linear} for details.)
\end{proof}

\begin{lemma}
The functional $J_{x,t}^{LQ}$ is uniformly convex and has a unique minimizer.
\end{lemma}

\begin{proof}
Observe that
$$
\bb{E}[\langle Q(s)X(s),\bar X(s)\rangle] = \bb{E}[\bb{E}[\langle Q(s)X(s),\bar X(s)\rangle |\sr{F}_s^0]]
= \bb{E}[\langle\bb{E}[ Q(s)X(s)|\sr{F}_s^0],\bar X(s)\rangle]
= \bb{E}[Q(s)\bar X(s),\bar X(s)\rangle],
$$
so that
\begin{multline*}
\bb{E}[\langle Q(s)X(s),X(s)\rangle + \langle \bar Q(s)\bar X(s),\bar X(s)]
\\
= \bb{E}[\langle Q(s)(X(s)-\bar X(s)),X(s)-\bar X(s)\rangle + \langle (Q(s) + \bar Q(s))\bar X(s),\bar X(s)].
\end{multline*}
In like manner, we have
\begin{multline*}
\bb{E}[\langle R(s)v(s),v(s)\rangle + \langle \bar R(s)\bar v(s),\bar v(s)]
\\
= \bb{E}[\langle R(s)(v(s)-\bar v(s)),v(s)-\bar v(s)\rangle + \langle (R(s) + \bar R(s))\bar v(s),\bar v(s)].
\end{multline*}
and
\begin{multline*}
\bb{E}[\langle S(s)X(s),v(s)\rangle + \langle \bar S(s)\bar X(s),\bar v(s)]
\\
= \bb{E}[\langle S(s)(X(s)-\bar X(s)),v(s)-\bar v(s)\rangle + \langle (S(s) + \bar S(s))\bar X(s),\bar v(s)].
\end{multline*}
Strict convexity now follows from Assumption \ref{coeffs}.
The existence and uniqueness of the minimizer follows from the weak lower semicontinuity of the functional.

\end{proof}

\subsection{Optimality conditions}

\begin{proposition} \label{mftc solution}
Suppose $v$ is an optimal control minimizing the functional $J_{X,t}(v)$, with corresponding trajectory $X(s)$ (the solution of \eqref{lq dynamics}).
Then there exists a unique adapted solution $(Y,Z,Z_0)$ of the BSDE
\begin{equation}
\label{adjoint}
\left\{
\begin{array}{rl}
dY(s) = & -\biggl(A^T(s)Y(s) + \bar A^T(s)\bar Y(s) + C^T(s)Z(s) + \bar C^T \bar Z(s) + F^T(s)Z_0(s) + \bar{F}^T(s)\bar Z_0(s)
\\
&  + Q(s)X(s) + \bar Q(s)\bar X(s) + S^T(s)v(s) + \bar S^T(s)\bar v(s) + q(s) + \bar q(s) \biggr)ds
\\
& + Z(s)dW(s) + Z_0(s)dW_0(s),
\ \ \ \ s \in [0,T],
\vspace{.2cm}
\\
Y(T) = & HX(T) + \bar H\bar X(T)
\end{array}
\right.
\end{equation}
satisfying the coupling condition
\begin{multline}
\label{coupling}
B^T(s)Y(s) + \bar{B}^T(s)\bar Y(s) + D^T(s)Z(s) + \bar{D}^T(s)\bar Z(s)
+ G^T(s)Z_0(s) + \bar{G}^T(s)\bar Z_0(s)
\\
 + R(s)v(s) + \bar R(s) \bar v(s)
 + S(s)X(s) + \bar S(s) \bar X(s) + r(s) + \bar r(s) = 0,
\ \ \ \ \ s \in [0,T], \ \text{a.s.}
\end{multline}
Here, as usual, $\bar Y(s) := \bb{E}[Y(s)|\sr{F}_s^0], \bar Z(s) := \bb{E}[Z(s)|\sr{F}_s^0],$ and $\bar Z_0(s) := \bb{E}[Z_0(s)|\sr{F}_s^0]$.

Conversely, suppose $(X,v,Y,Z)$ is an adapted solution to the forward-backward system \eqref{lq dynamics},\eqref{adjoint}.
Then $v$ is the optimal control minimizing $J_{X,t}(v)$, and $X(s)$ is the optimal trajectory.
\end{proposition}

\begin{proof}
The G\^ateaux derivative of $J^{LQ}_{x,t}(v)$ is 
\begin{multline} \label{Gateaux derivative}
\left.\frac{d}{dh}J^{LQ}_{x,t}(v + h\tilde v)\right|_{h=0}
= 2\bb{E}\biggl\{\int_t^T \bigl[\langle Q(s)X(s),\tilde X(s) \rangle  + \langle \bar Q(s)\bar X(s),\bar {\tilde X}(s) \rangle  
+ \langle R(s)v(s),\tilde v(s) \rangle 
\\
+ \langle \bar R(s) \bar v(s),\bar{\tilde v}(s) \rangle 
+ \langle S(s)X(s),\tilde v(s) \rangle
+ \langle S(s)\tilde X(s), v(s) \rangle
+ \langle \bar S(s)\bar X(s),\bar{\tilde v}(s) \rangle
+ \langle \bar S(s)\bar {\tilde X}(s),\bar{ v}(s) \rangle
\\
+ \langle q(s),\tilde X(s) \rangle + \langle \bar q(s),\bar{\tilde X}(s) \rangle 
+ \langle r(s),\tilde v(s) \rangle + \langle \bar r(s),\bar{\tilde v}(s) \rangle 
\bigr]ds
+ \langle HX(T),\tilde X(T) \rangle + \langle \bar H \bar X(T),\bar{\tilde X}(T) \rangle \biggr\}
\end{multline}
where $\tilde X$ is the solution of \eqref{lq dynamics} with $v$ replaced by $\tilde v$ and $X$ replaced by $0$.
If $(X,v)$ is optimal, then we get the optimality condition
\begin{multline}
\label{first order condition}
\bb{E}\biggl\{\int_t^T \bigl[\langle Q(s)X(s) + \bar Q(s)\bar X(s) + S^T(s)v(s) + \bar S^T(s)\bar v(s) + q(s) + \bar q(s),\tilde X(s) \rangle
\\
+ \langle R(s)v(s) + \bar R(s) \bar v(s) + S(s)X(s) + \bar S(s)\bar X(s) + r(s) + \bar r(s),\tilde v(s) \rangle \bigr]ds
\\
+ \langle HX(T) + \bar H \bar X(T),\tilde X(T) \rangle \biggr\} = 0
\end{multline}

Now by \cite{buckdahn2009mean,buckdahn2009mean-pde} we have a solution to the McKean-Vlasov type BSDE \eqref{adjoint}.
By the It\^o formula,
\begin{multline}
\bb{E}\langle Y(T),\tilde X(T) \rangle 
\\
= \bb{E} \int_t^T \langle B^T(s)Y(s) + \bar{B}^T(s)\bar Y(s) + D^T(s)Y(s) + \bar{D}^T(s)\bar Z(s) + G^T(s)Z_0(s) + \bar{G}^T(s)\bar Z_0(s),\tilde v(s)\rangle ds
\\
- \bb{E} \int_t^T \langle \tilde{X}(s),Q(s)X(s) + \bar Q(s)\bar X(s) + S^T(s)v(s) + \bar S^T(s)\bar v(s) + q(s) + \bar q(s)\rangle ds,
\end{multline}
which by using the optimality condition \eqref{first order condition} and the fact that $Y(T) =HX(T) + \bar H \bar X(T)$ becomes
\begin{multline}
0
= \bb{E} \int_t^T \langle B^T(s)Y(s) + \bar{B}^T(s)\bar Y(s) + D^T(s)Y(s) + \bar{D}^T(s)\bar Z(s) + G^T(s)Z_0(s) + \bar{G}^T(s)\bar Z_0(s),\tilde v(s)\rangle ds
\\
+ \bb{E} \int_t^T \langle R(s)v(s) + \bar R(s)\bar v(s) + S(s)X(s) + \bar S(s)\bar X(s) + r(s) + \bar r(s),\tilde v(s)\rangle ds.
\end{multline}
Since $\tilde v$ is arbitrary, we obtain the coupling condition \eqref{coupling}, as desired.

To prove the converse, it suffices to note that our assumptions imply that $J^{LQ}_{x,t}(\cdot)$ is strictly convex.
Then, given a solution $(X(s),v(s),Y(s),Z(s),Z_0(s))$ to the system \eqref{lq dynamics},\eqref{adjoint}, we know from \eqref{Gateaux derivative} that the G\^ateaux derivative of $J^{LQ}_{x,t}$ at $v$ is zero, which implies $v$ is the minimizer, as desired.
\end{proof}

\subsection{Riccati equations}

In order to find explicit solutions of the mean field type control problem, we derive a system of Riccati equations.
We use a technique developed by Yong in \cite{yong2013linear}.
We suppose
$$
Y(s) = P(s)(X(s)-\bar X(s)) + \Pi(s)\bar X(s) + \phi(s)
$$
where $P$ and $\Pi$ are $\s{S}^d$-valued processes such that
$$
P(T) = H, \ \Pi(T) = H + \bar H.
$$
and $\phi(s)$ is an $\bb{R}^d$-valued process.
Note that $P,\Pi,$ and $\phi$ are deterministic. Our goal is to derive a system of ordinary differential equations governing their evolution (backwards) in time.

By taking conditional expectation we have
$$
\bar Y(s) = \Pi(s)\bar X(s) + \phi(s)
\ \ \ \text{and} \ \ \
Y(s) - \bar Y(s) = P(s)(X(s) - \bar X(s)).
$$
Now
\begin{equation}
d\bar X = \left\{(A + \bar{A})\bar X + (B + \bar{B})\bar v \right\}ds
+ \left\{(F + \bar{F})\bar X + (G + \bar{G})\bar v\right\}dW_0,
\end{equation}
which implies
\begin{multline}
d(X-\bar X) = \left\{A(X - \bar X) + B(v - \bar v) \right\}ds
\\
+ \left\{C(X-\bar X) + (C+\bar{C})\bar X + D(v-\bar v) + (D+\bar{D})\bar v\right\}dW
+ \left\{F(X -\bar X) + G(v - \bar v)\right\}dW_0.
\end{multline}
Now recall that
\begin{multline} \label{adjoint dynamics}
dY = -\biggl(A^TY + \bar A^T\bar Y + C^TZ + \bar C^T \bar Z
 + F^TZ_0 + \bar{F}^T\bar Z_0 + QX + \bar Q\bar X
 \\
 + S^Tv + \bar S^T \bar v + q + \bar q \biggr)ds
 + ZdW + Z_0dW_0
 \\
= -\biggl(A^T(Y-\bar Y)
 + (A^T + \bar A^T)\bar Y
  + C^T(Z-\bar Z) + (C^T+\bar C^T) \bar Z
 + F^T(Z_0-\bar Z_0) + (F^T+\bar{F}^T)\bar Z_0
 \\ + Q(X-\bar X) + (Q +\bar Q)\bar X 
 + S^Tv + \bar S^T \bar v + q + \bar q \biggr)ds
 + ZdW + Z_0dW_0.
\end{multline}
On the other hand,
\begin{multline} \label{d(Y-EY)}
d(Y-\bar Y) = \dot P(X-\bar X)ds
+ Pd(X-\bar X)
\\
= \left\{\dot P(X-\bar X)
+ PA(X - \bar X) + PB(v - \bar v) \right\}ds
\\
+ P\biggl\{C(X-\bar X) + (C+\bar{C})\bar X +
 D(v-\bar v) + (D+\bar{D})\bar v\biggr\}dW
+ P\left\{F(X -\bar X) + G(v - \bar v)\right\}dW_0
\end{multline}
while
\begin{multline} \label{dEY}
d\bar Y = (\dot \phi + \dot \Pi\bar X)ds + \Pi d\bar X
\\
= \left\{\dot \phi + \dot \Pi\bar X + \Pi(A + \bar{A})\bar X + \Pi(B + \bar{B})\bar v \right\}ds
+ \Pi\left\{(F + \bar{F})\bar X + (G + \bar{G})\bar v\right\}dW_0.
\end{multline}
Note that $dY = d(Y-\bar Y) + d\bar Y$.
By comparing the diffusion terms, we get
\begin{equation}
Z = P\biggl\{C(X-\bar X) + (C+\bar{C})\bar X +
D(v-\bar v) + (D+\bar{D})\bar v\biggr\}
\end{equation}
and
\begin{equation}
Z_0 = P\left\{F(X -\bar X) + G(v - \bar v)\right\}
+ \Pi\left\{(F + \bar{F})\bar X + (G + \bar{G})\bar v\right\}
\end{equation}
which imply
\begin{equation} \label{EZ}
\bar Z = P\left\{ (C+\bar{C})\bar X +
 (D+\bar{D})\bar v\right\},
\end{equation}
\begin{equation} \label{Z-EZ}
Z-\bar Z = P\left\{C(X-\bar X) + 
 D(v-\bar v) \right\},
\end{equation}
\begin{equation} \label{EZ_0}
\bar Z_0 = \Pi\left\{(F + \bar{F})\bar X + (G + \bar{G})\bar v\right\},
\end{equation}
and
\begin{equation} \label{Z-EZ_0}
Z_0-\bar Z_0 = P\left\{F(X -\bar X) + G(v - \bar v)\right\}.
\end{equation}
Next we use the coupling condition \eqref{coupling} to find a formula for $v(s)$.
We have
\begin{multline}
0 = B^T(Y-\bar Y) + (B^T+\bar{B}^T)\bar Y 
+ D^T(Z-\bar Z) + (D^T+\bar{D}^T)\bar Z
+ G^T(Z_0-\bar Z_0) + (G^T+\bar{G}^T)\bar Z_0 
\\+ R(v-\bar v) + (R+\bar R)\bar v
+ S(X - \bar X) + (S + \bar S)\bar X + r + \bar r
\\
= B^TP(X-\bar X) + (B^T+\bar{B}^T)\Pi\bar X + (B^T+\bar{B}^T)\phi
+ D^T P\left\{C(X-\bar X) + D(v-\bar v) \right\} 
 \\+ (D^T+\bar{D}^T)P\left\{ (C+\bar{C})\bar X + (D+\bar{D})\bar v\right\}
+ G^T P\left\{F(X -\bar X) + G(v - \bar v)\right\}
\\+ (G^T+\bar{G}^T)\Pi\left\{(F + \bar{F})\bar X + (G + \bar{G})\bar v\right\} 
+ R(v-\bar v) + (R+\bar R)\bar v
+ S(X - \bar X) + (S + \bar S)\bar X + r + \bar r
\\
= \Lambda_0(X-\bar X)
+ \Lambda_1
\bar X
+ \Sigma_0(v-\bar v)
+ \Sigma_1
\bar v
+ (B^T+\bar{B}^T)\phi + r + \bar r
\end{multline}
with
\begin{align*}
\Lambda_0 &= B^TP+ D^TPC + G^TPF + S,\\
\Lambda_1 &=(B^T+\bar{B}^T)\Pi+ (D^T+\bar{D}^T)P(C+\bar{C}) + (G^T+\bar{G}^T)\Pi(F + \bar{F}) + S + \bar S,\\
\Sigma_0 &= D^TPD + R,\\
\Sigma_1 &= (D^T+\bar{D}^T)P(D+\bar{D}) + (G^T+\bar{G}^T)\Pi(G + \bar{G}) + (R+\bar R).
\end{align*}
Taking conditional expectation, we deduce
\begin{equation}
\Sigma_1(s)\bar v(s) + \Lambda_1(s)\bar X(s)
+ r(s) + \bar r(s) + (B^T(s)+\bar{B}^T(s))\phi(s)
 = 0
\end{equation}
so that, assuming $\Sigma_1(s)$ is invertible,
\begin{equation} \label{Ev}
\bar v(s) = -\Sigma_1(s)^{-1}(\Lambda_1(s)
\bar X(s) + r(s) + \bar r(s) + (B^T(s)+\bar{B}^T(s))\phi(s)).
\end{equation}
Assuming $\Sigma_0(s)$ is also invertible, we therefore have
\begin{multline} \label{v}
v(s) = v(s)-\bar v(s) + \bar v(s)
\\
= -\Sigma_0(s)^{-1}(\Lambda_0(s)(X(s)-\bar X(s))
+ \Lambda_1(s)
\bar X(s) + \Sigma_1(s)
\bar v(s)
+ r(s) + \bar r(s) + (B^T(s)+\bar{B}^T(s))\phi(s))
+ \bar v(s)
\\
= -\Sigma_0(s)^{-1}\Lambda_0(s)(X(s)-\bar X(s))
-\Sigma_1(s)^{-1}(\Lambda_1(s)
\bar X(s) + r(s) + \bar r(s) + (B^T(s)+\bar{B}^T(s))\phi(s)).
\end{multline}

Now we compare the drift terms from \eqref{adjoint dynamics} to those of \eqref{d(Y-EY)} and \eqref{dEY}.
Using the relations \eqref{EZ},\eqref{Z-EZ},\eqref{EZ_0},\eqref{Z-EZ_0},\eqref{Ev}, and \eqref{v} proved above, we get
\begin{multline}
0
= \biggl(\dot P
+A^TP + PA + C^TPC + F^T P F + Q
- (PB + C^TPD + F^TPG + S^T)\Sigma_0^{-1}\Lambda_0 \biggr)(X-\bar X)
\\
+ \biggl(\dot \Pi
+(A^T+\bar A^T)\Pi
 + \Pi(A + \bar{A})
 +(C^T+\bar C^T) P(C+\bar{C})
 +(F^T+\bar{F}^T)\Pi(F + \bar{F})
 + Q +\bar Q
 \\
 - (\Pi(B+\bar B) + (C^T + \bar C^T)P(D + \bar D) + (F^T + \bar F^T)\Pi(G+\bar G) + S^T + \bar S^T)\Sigma_1^{-1}\Lambda_1\biggr)\bar X
 \\ 
 + \dot \phi 
 - (\Pi(B+\bar B) + (C^T + \bar C^T)P(D + \bar D) + (F^T + \bar F^T)\Pi(G+\bar G) + S^T + \bar S^T)\Sigma_1^{-1}(r+ \bar r) + q + \bar q.
\end{multline}
We deduce that $P$ and $\Pi$ should satisfy the following Riccati equations:
\begin{multline}
\label{P equation}
\left\{
\begin{array}{l}
\dot P
+A^TP
+ PA +C^TPC + F^TPF
+Q \vspace{.2cm}
\\ \ \
- (PB+C^TPD+F^TPG+S^T)(D^TPD + R)^{-1}(B^TP + D^TPC+ G^TPF + S) = 0, \vspace{.2cm}
\\
P(T) = H
\end{array}\right.
\end{multline}
and
\begin{multline}
\label{Pi equation}
\left\{
\begin{array}{l}
\dot \Pi
+(A^T+\bar A^T)\Pi
 + \Pi(A + \bar{A})
 +(C^T+\bar C^T) P(C+\bar{C})
  +(F^T+\bar{F}^T)\Pi(F + \bar{F})
   +(Q +\bar Q) \vspace{.2cm}
\\ \ \ - \biggl(\Pi(B + \bar{B})+(C^T+\bar C^T) P(D+\bar{D})+(F^T+\bar{F}^T)\Pi(G + \bar{G}) + S^T + \bar S^T\biggr)\Sigma_1^{-1}
\\ \ \ \ \ \ \
\cdot \biggl((B + \bar{B})^T\Pi+ (D+\bar{D})^TP(C+\bar C)+(G + \bar{G})^T\Pi(F+\bar{F}) + S + \bar S\biggr)  = 0, \vspace{.2cm}
\\ \ \ \Sigma_1 = (D^T+\bar{D}^T)P(D+\bar{D}) + (G^T+\bar{G}^T)\Pi(G + \bar{G}) + (R+\bar R), \vspace{.2cm}
\\
\Pi(T) = H + \bar H.
\end{array}\right.
\end{multline}
Once we have $P,\Pi$ solutions to \eqref{P equation},\eqref{Pi equation}, respectively, we set
\begin{equation*}
\phi(s) = \int_t^s \biggl\{(\Pi(B+\bar B) + (C^T + \bar C^T)P(D + \bar D) + (F^T + \bar F^T)\Pi(G+\bar G) + S^T + \bar S^T)\Sigma_1^{-1}(r+ \bar r) + q + \bar q \biggr\}d\tau.
\end{equation*}
A standard reference on optimal control, e.g.~\cite{yong1999stochastic}, suffices to show that under the given assumptions \eqref{P equation} has a unique solution, which in addition is symmetric.
To see that \eqref{Pi equation} has a unique solution, we note that $(C^T+\bar C^T) P(C+\bar{C})$ is also a symmetric matrix, and therefore by the same reference we can deduce there exists a unique solution of \eqref{Pi equation}, which is also symmetric.

We can deduce the dynamics of the optimal trajectory using \eqref{v}.
\begin{multline} \label{X}
dX = \biggl\{A(X-\bar X) + (A+\bar{A})\bar X
 + B(v-\bar v) + (B +\bar{B})\bar v \biggr\}ds
 \\
+ \biggl\{C(X-\bar X) + (C+\bar{C})\bar X
 + D(v-\bar v) + (D+\bar{D})\bar v\biggr\}dW
\\
+ \biggl\{F(X-\bar X) + (F+\bar{F})\bar X 
+ G(v-\bar v) + (G+\bar{G})\bar v\biggr\}dW_0
\\ = \biggl\{(A- B\Sigma_0^{-1}\Lambda_0)(X-\bar X)
 + (A+\bar{A}- (B +\bar{B})\Sigma_1^{-1}\Lambda_1)\bar X 
 - (B +\bar{B})\Sigma_1^{-1}(r+\bar r)
 \biggr\}ds
\\
+ \biggl\{(C- D\Sigma_0^{-1}\Lambda_0)(X-\bar X)
 + (C+\bar{C}-(D+\bar{D})\Sigma_1^{-1}\Lambda_1)\bar X 
 -(D+\bar{D})\Sigma_1^{-1}(r+\bar r)
 \biggr\}dW
\\
+ \biggl\{(F- G\Sigma_0^{-1}\Lambda_0)(X-\bar X)
 + (F+\bar{F}-(G+\bar{G})\Sigma_1^{-1}\Lambda_1)\bar X 
 -(G+\bar{G})\Sigma_1^{-1}(r+\bar r)
 \biggr\}dW_0.
\end{multline}
A formula for the process $Z$ can also be deduced from \eqref{Z-EZ}, \eqref{EZ}, and \eqref{v}:
\begin{multline}
Z = Z-\bar Z + \bar Z
 = P\left\{C - 
  D\Sigma_0^{-1}\Lambda_0 \right\}(X-\bar X)
  \\
  + P\left\{ (C+\bar{C}) -
   (D+\bar{D})\Sigma_1^{-1}\Lambda_1\right\}\bar X - P(D+\bar D)\Sigma_1^{-1}(r+\bar r),
\end{multline}
and for $Z_0$ we get the following by using \eqref{Z-EZ_0} and \eqref{EZ_0}:
\begin{multline}
Z_0 = Z-\bar Z + \bar Z_0 
= P\left\{F - 
 G\Sigma_0^{-1}\Lambda_0 \right\}(X-\bar X)
 \\
+ \Pi\left\{(F + \bar{F}) - (G + \bar{G})\Sigma_1^{-1}\Lambda_1\right\}\bar X
- \Pi(G + \bar{G})\Sigma_1^{-1}(r+ \bar r).
\end{multline}
We summarize our results here:
\begin{theorem} \label{riccati solver}
There exists a unique solution $P,\Pi$ to the pair of Riccati equations \eqref{P equation} and \eqref{Pi equation}, where $P,\Pi$ are both $\s{S}^d$-valued deterministic processes.
Moreover, the unique optimal trajectory for Problem \eqref{LQ optimal control problem} is given by the solution to the SDE
\begin{multline} \label{closed loop dynamics}
dX  = \biggl\{(A- B\Sigma_0^{-1}\Lambda_0)(X-\bar X)
 + (A+\bar{A}- (B +\bar{B})\Sigma_1^{-1}\Lambda_1)\bar X 
 - (B +\bar{B})\Sigma_1^{-1}(r+\bar r)
 \biggr\}ds
\\
+ \biggl\{(C- D\Sigma_0^{-1}\Lambda_0)(X-\bar X)
 + (C+\bar{C}-(D+\bar{D})\Sigma_1^{-1}\Lambda_1)\bar X 
 -(D+\bar{D})\Sigma_1^{-1}(r+\bar r)
 \biggr\}dW
\\
+ \biggl\{(F- G\Sigma_0^{-1}\Lambda_0)(X-\bar X)
 + (F+\bar{F}-(G+\bar{G})\Sigma_1^{-1}\Lambda_1)\bar X 
 -(G+\bar{G})\Sigma_1^{-1}(r+\bar r)
 \biggr\}dW_0.
\end{multline}
with
\begin{align*}
\Lambda_0 &= B^TP+ D^TPC + G^TPF + S,\\
\Lambda_1 &=(B^T+\bar{B}^T)\Pi+ (D^T+\bar{D}^T)P(C+\bar{C}) + (G^T+\bar{G}^T)\Pi(F + \bar{F}) + S + \bar S,\\
\Sigma_0 &= D^TPD + R,\\
\Sigma_1 &= (D^T+\bar{D}^T)P(D+\bar{D}) + (G^T+\bar{G}^T)\Pi(G + \bar{G}) + (R+\bar R).
\end{align*}
The SDE \eqref{closed loop dynamics} has a unique solution.
The optimal control is given by
\begin{equation} \label{the optimal control}
v(s) 
= -\Sigma_0(s)^{-1}\Lambda_0(s)(X(s)-\bar X(s))
-\Sigma_1(s)^{-1}(\Lambda_1(s)
\bar X(s) + r(s) + \bar r(s) + (B^T(s)+\bar{B}^T(s))\phi(s)).
\end{equation}
If we define the adjoint processes
\begin{align*}
Y(s) &= P(s)(X(s)-\bar X(s)) + \Pi(s)\bar X(s) + \phi(s),\\
Z(s) &= P\left\{C - 
  D\Sigma_0^{-1}\Lambda_0 \right\}(X-\bar X)
  + P\left\{ (C+\bar{C}) -
   (D+\bar{D})\Sigma_1^{-1}\Lambda_1\right\}\bar X - P(D+\bar D)\Sigma_1^{-1}(r+\bar r),\\
Z_0(s) &= P\left\{F - 
 G\Sigma_0^{-1}\Lambda_0 \right\}(X-\bar X)
+ \Pi\left\{(F + \bar{F}) - (G + \bar{G})\Sigma_1^{-1}\Lambda_1\right\}\bar X
- \Pi(G + \bar{G})\Sigma_1^{-1}(r+ \bar r)
\end{align*}
where
\begin{equation*}
\phi(s) = \int_t^s \biggl\{(\Pi(B+\bar B) + (C^T + \bar C^T)P(D + \bar D) + (F^T + \bar F^T)\Pi(G+\bar G) + S^T + \bar S^T)\Sigma_1^{-1}(r+ \bar r) + q + \bar q \biggr\}d\tau,
\end{equation*}
then the quintuple $(X,v,Y,Z,Z_0)$ is an adapted solution to the mean field FBSDE \eqref{lq dynamics},\eqref{adjoint}.
\end{theorem}

\section{The Mean Field Game} \label{sec:mfg}

In this section, we consider the problem of Nash equilibrium rather than optimal control.
We modify the dynamics \eqref{lq dynamics} as follows.
Let $\bar \xi(s)$ and $\bar \psi(s)$ be given processes adapted to the filtration $\{\sr{F}_s^0\}_{s \geq t}$.
Consider
\begin{multline} \label{lq mfg dynamics}
dX(s) = \left\{A(s)X(s) + \bar{A}(s)\bar \xi(s) + B(s)v(s) + \bar{B}(s)\bar \psi(s) \right\}ds
\\
+ \left\{C(s)X(s) + \bar{C}(s)\bar \xi(s) + D(s)v(s) + \bar{D}(s)\bar \psi(s)\right\}dW(s)
\\
+ \left\{F(s)X(s) + \bar{F}(s)\bar \xi(s) + G(s)v(s) + \bar{G}(s)\bar \psi(s)\right\}dW_0(s), \ \ X(t) = x
\end{multline}
and the objective functional
\begin{multline} \label{mfg objective functional}
J^{mfg}_{x,t}(v;\bar \xi,\bar \psi) = \bb{E}\biggl\{\int_t^T \bigl[\langle Q(s)X(s),X(s) \rangle  + 2\langle \bar Q(s)\bar \xi(s),X(s) \rangle  
+ \langle R(s)v(s),v(s) \rangle + 2\langle \bar R(s) \bar \psi(s),v(s) \rangle
\\
+ 2\langle S(s)X(s),v(s) \rangle + 2\langle \bar S_1(s) \bar \xi(s),v(s) \rangle + 2\langle \bar S_2(s) X(s),\bar \psi(s) \rangle
\\
+ 2\langle q(s),X(s)\rangle + 2\langle \bar q(s), \bar \xi(s)\rangle
+ 2\langle r(s),v(s) \rangle + 2\langle \bar r(s), \bar \psi(s) \rangle \bigr]ds
+ \langle HX(T),X(T) \rangle + 2\langle \bar H \bar \xi(T),X(T) \rangle \biggr\}.
\end{multline}
The goal is to find a process $\hat v(s)$ such that, given the process $\hat X(s)$ generated by $\hat v$, we have
\begin{equation}
\label{Nash equilibrium condition}
J_{X,t}^{NE}(\hat v;\bar \xi,\bar \psi) = \inf_{v} J_{X,t}^{NE}(v;\bar \xi,\bar \psi) \ \ \text{and} \ \ 
\bb{E}[\hat X(s)|\sr{F}_s^0] = \bar \xi, \ 
\bb{E}[\hat v(s)|\sr{F}_s^0] = \bar \psi.
\end{equation}
Such a process $\hat v$ is called a mean field Nash equilibrium.
Note that, rather than an optimizer, we are seeking a \textit{fixed point} of the map
$$
v \mapsto (X,v) \mapsto (\bar \xi,\bar \psi) \mapsto \hat v,
$$
where any given control $v$ generates a state process $X$, $(X,v)$ generate processes $(\bar \xi,\bar \psi)$ by taking conditional expectation with respect to the common noise, and $\hat v$ is an optimal control with respect to these given processes.

We make similar assumptions as before:
\begin{assumption}
The following are the assumption on the coefficient matrices for the mean field game:
\begin{enumerate}
\item $A,\bar A,C, \bar C, F, \bar F \in L^\infty(0,T;\bb{R}^{d \times d})$
\item $B, \bar B,D, \bar D, G, \bar G \in L^\infty(0,T;\bb{R}^{d \times m})$
\item $Q, \bar Q \in L^\infty(0,T;\s{S}^d), \ R,\bar R \in L^\infty(0,T;\s{S}^m), \ H,\bar H \in \s{S}^d$
\item $H \geq 0$, and for some $\delta_1 \geq 0, \delta_2 > 0$, $Q \geq \delta_1 I$ and $R \geq \delta_2 I$
\item $S,\bar S_1, \bar S_2 \in L^\infty(0,T;\bb{R}^{m \times d}); q, \bar q \in L^\infty(0,T;\bb{R}^d); r, \bar r \in L^\infty(0,T;\bb{R}^m)$
\item $\|S\|_\infty^2 < \delta_1 \delta_2$ if $\delta_1 \neq 0$, $S = 0$ otherwise.
\end{enumerate}
\end{assumption}

Our goal is not to find the most general conditions under which we may solve this fixed point problem (cf.~\cite{ahuja2016wellposedness,carmona2014mean}).
Indeed, the fact that the conditional expectation \textit{of the control variable} appears in the objective functional (in equilibrium) seems to complicate matters considerably.
In the first part of this section, we will give a brief discussion of where the difficulty lies.
In the following subsections, we will focus more on the case which is of most interest to us, namely when the Nash equilibrium can be computed by finding an optimizer to a mean field type control problem.

Let us first see that, for given processes $(\bar \xi,\bar \psi)$, there is indeed an optimal control $\hat v$.
Using the same arguments as in the previous section, we obtain the following characterization.
\begin{proposition} \label{mfg solution}
For a given pair $(\bar \xi, \bar \pi)$, there exists a unique optimal control $v$ minimizing the functional $J_{X,t}^{NE}(v;\bar \xi,\bar \psi)$.
Furthermore, let $X(s)$ be the corresponding trajectory (the solution of \eqref{lq mfg dynamics} with control $v(s)$).
Then there exists a unique adapted solution $(Y,Z,Z_0)$ of the BSDE
\begin{equation}
\label{adjoint mfg}
\left\{
\begin{array}{rl}
dY(s) = & -\biggl(A^T(s)Y(s) + C^T(s)Z(s) + F^T(s)Z_0(s) + 
\\
&  + Q(s)X(s) + \bar Q(s)\bar \xi(s) + S^T(s)v(s) + \bar S_2^T(s)\bar \psi(s) + q(s) \biggr)ds
\\
& + Z(s)dW(s) + Z_0(s)dW_0(s),
\ \ \ \ s \in [0,T],
\vspace{.2cm}
\\
Y(T) = & HX(T) + \bar H\bar \xi(T)
\end{array}
\right.
\end{equation}
satisfying the coupling condition
\begin{multline}
\label{coupling mfg}
B^T(s)Y(s) + D^T(s)Z(s)
+ G^T(s)Z_0(s)
\\
 + R(s)v(s) + \bar R(s) \bar \psi(s)
 + S(s)X(s) + \bar S_1(s) \bar \xi(s) + r(s) = 0,
\ \ \ \ \ s \in [0,T], \ \text{a.s.}
\end{multline}

Conversely, suppose $(X,v,Y,Z)$ is an adapted solution to the forward-backward system \eqref{lq mfg dynamics},\eqref{adjoint mfg} and coupling condition \eqref{coupling mfg}.
Then $v$ is the optimal control minimizing $J_{X,t}^{NE}(v;\bar \xi,\bar \psi)$, and $X(s)$ is the optimal trajectory.
If, in addition, we have $\bb{E}[X(s)|\sr{F}_s^0] = \bar \xi(s)$ and $\bb{E}[v(s)|\sr{F}_s^0] = \bar \psi(s)$ then $v(s)$ is a mean field Nash equilibrium.
\end{proposition}

Proposition \ref{mfg solution} can be seen as an abstract condition for the solvability of the fixed point problem.
However, one would like to have a more concrete criterion giving existence of a mean field Nash equilibrium.
Let us attempt to follow the discussion in \cite{bensoussan2011linear} and find out where the difficulty lies.

Suppose there exists a mean field Nash equilibrium $v(s)$.
Then $\bar X(s) = \bb{E}[X(s)| \sr{F}_s^0]$ satisfies the dynamics
\begin{equation} \label{bar X dynamics}
d\bar X = \left\{(A + \bar{A})\bar X + (B + \bar{B})\bar v \right\}ds
+ \left\{(F + \bar{F})\bar X + (G + \bar{G})\bar v\right\}dW_0, \ \ \bar X(t) = \bar X
\end{equation}
while $\bar Y$ satisfies
\begin{equation} \label{bar Y dynamics}
d\bar Y =  -\biggl(A^T \bar Y + C^T \bar Z + F^T \bar Z_0 + (Q + \bar Q)\bar X 
  + (S^T + \bar S_2^T)\bar v + q \biggr)ds + \bar Z_0dW_0,
\ \ \
\bar Y(T) =  (H + \bar H)\bar \xi(T)
\end{equation}
and we have the coupling condition
\begin{equation} \label{bar v coupling}
B^T\bar Y + D^T\bar Z + G^T \bar Z_0 
 + (R + \bar R) \bar v
 + (S + \bar S_1) \bar X + r = 0.
\end{equation}
Conversely, suppose that the system \eqref{bar X dynamics},\eqref{bar Y dynamics},\eqref{bar v coupling} has a solution which we denote $(\bar \xi,\bar \eta,\bar \psi)$ (corresponding to $(\bar X,\bar Y,\bar v)$).
Then we let $v$ be the optimal control minimizing $J_{X,t}^{NE}(v;\bar \xi,\bar \psi)$, and let $X(s)$ be the state solving the dynamics \eqref{lq mfg dynamics} and $Y(s)$ the adjoint state solving \eqref{adjoint mfg}.
Note that, by Proposition \ref{mfg solution}, the coupling condition \eqref{coupling mfg} is satisfied.
Now, if we knew that $\bar v(s) := \bb{E}[v(s)|\sr{F}_s^0] = \bar \psi(s)$, then we would see that $(\bar X,\bar Y,\bar v)$ is also a solution to the system \eqref{bar X dynamics},\eqref{bar Y dynamics},\eqref{bar v coupling}, and it would not be difficult to see that therefore $\bar X = \bar \xi$ and $\bar Y = \bar \eta$ as well.
Thus we would have that $v$ is a Nash equilibrium.
However, this is a nontrivial criterion on $v$, since there is nothing which obviously connects the control problem of minimizing \eqref{mfg objective functional} with the system \eqref{bar X dynamics},\eqref{bar Y dynamics},\eqref{bar v coupling}.
So this fails to be an appropriate criteria for determining the existence of Nash equilibrium.

On the other hand, if we take $\bar v$ out of the problem, i.e.~if we set $\bar B = \bar D = \bar G = \bar R = \bar S_2 = 0$, then we obtain, as in \cite{bensoussan2011linear}, a necessary and sufficient condition for the existence of Nash equilibrium.
This is summarized in the following proposition, whose proof is essentially the same as that of \cite[Theorem III.4]{bensoussan2011linear}.
\begin{proposition}
Let $\bar B = \bar D = \bar G = \bar R = \bar S_2 = 0$.
Then there exists a mean-field Nash equilibrium for the objective functional $J_{X,t}^{NE}$, given in \eqref{mfg objective functional}, if and only if there exists a solution to the forward-backward system of stochastic differential equations given by \eqref{bar X dynamics},\eqref{bar Y dynamics}, and \eqref{bar v coupling}.
\end{proposition}
There is another condition on the coefficients, given in Section \ref{sec:mfg = mftc} below, which permits us to assert that there \textit{always} exists a unique mean field Nash equilibrium, namely when the mean field game corresponds to an optimal control problem.
While this is not the most general case, it is directly applicable to the economics example which motivated this work.
Moreover, as we will see, we need not give up all dependence on the conditional expectation of the control variable.

\subsection{When is a Mean Field Game equivalent to a Mean Field Type Control Problem?} \label{sec:mfg = mftc}

It is now well-known that, as pointed out in \cite{lasry07}, mean field Nash equilibria can be characterized--at least formally--as optimality conditions for mean field type control problems.
Often this is exploited to obtain results on existence and uniqueness of solutions to mean field games.
See, for instance, \cite{cardaliaguet2013weak,graber2014optimal,cardaliaguet2014mean,cardaliaguet2015second,benamou2016variational}.
Here we point out a condition under which the Mean Field Type Control Problem and the Mean Field Game are equivalent for the general linear-quadratic case.

\begin{proposition} \label{mfg = mftc}
Let $\bar A, \bar B, \bar C, \bar D, \bar F, \bar G, \bar q,$ and $\bar r$ all be zero.
Additionally, assume $\bar S = \bar S_1 = \bar S_2$.

Suppose $\hat v(s)$ is the optimal control for the linear-quadratic functional $J_{X,t}^{LQ}(v)$ defined by \eqref{mftc objective functional}, with corresponding optimal trajectory $\hat X(s)$ defined by \eqref{lq dynamics}.
Define $\bar \xi(s) := \bb{E}[\hat X(s)| \sr{F}_s^0]$ and $\bar \psi(s) := \bb{E}[\hat v(s)| \sr{F}_s^0]$.
Then $\hat v(s)$ is a mean field Nash equilibrium for $J_{X,t}^{NE}(v;\bar \xi,\bar \psi)$ defined by \eqref{mfg objective functional}.

Conversely, if $\hat v(s)$ is a mean field Nash equilibrium for $J_{X,t}^{NE}(v;\bar \xi,\bar \psi)$, then it is also an optimal control for $J_{X,t}^{LQ}(v)$.
\end{proposition}

\begin{proof}
It suffices to note that, under the given assumptions, the (forward-backward) systems of stochastic differential equations and their coupling conditions given by Propositions \ref{mftc solution} and \ref{mfg solution} are equivalent, once we have taken into account the equilibrium condition $\bar \xi(s) = \bb{E}[\hat X(s)| \sr{F}_s^0]$ and $\bar \psi(s) = \bb{E}[\hat v(s)| \sr{F}_s^0]$.
\end{proof}

We note the following simple corollary of Proposition \ref{mfg = mftc}.
\begin{corollary}
\label{existence of Nash equilibria}
Let $\bar A, \bar B, \bar C, \bar D, \bar F, \bar G, \bar q,$ and $\bar r$ all be zero.
Additionally, assume $\bar S = \bar S_1 = \bar S_2$.

Let $\bar \xi,\bar \psi$ be given $\s{F}_s^0$-adapted processes with values in $\bb{R}^d,\bb{R}^m$ respectively.
Then the linear-quadratic functional $J_{X,t}^{NE}(v;\bar \xi,\bar \psi)$ defined by \eqref{mfg objective functional} has a mean field Nash equilibrium if and only if $\bar \xi(s) = \bb{E}[\hat X(s)| \sr{F}_s^0]$ and $\bar \psi(s) = \bb{E}[\hat v(s)| \sr{F}_s^0]$, where $\bar v$ is the optimal control of $J_{X,t}^{LQ}(v)$ with corresponding optimal trajectory $\hat X(s)$ defined by \eqref{lq dynamics}.
Moreover, this equilibrium is unique.
\end{corollary}

Note that the objective functionals $J_{X,t}^{LQ}$ and $J_{X,t}^{NE}$ given in \eqref{mftc objective functional} and \eqref{mfg objective functional}, respectively, are not precisely the same under the assumptions of Proposition \ref{mfg = mftc}.
In other words, even though $\hat v$ solves both a fixed point Nash equilibrium and an mean field type control problem, the costs are different.
Indeed, suppose the assumptions of Proposition \ref{mfg = mftc} hold, and take $\hat v$ to be a mean field Nash equilibrium with corresponding trajectory $\hat X$.
Define $\bar \xi(s) := \bb{E}[\hat X(s)| \sr{F}_s^0]$ and $\bar \psi(s) := \bb{E}[\hat v(s)| \sr{F}_s^0]$.
By the proposition, $\hat v$ is the minimizer of both $J^{NE}_{X,t}(v;\bar \xi,\bar \psi)$ and $J^{LQ}_{x,t}(v)$.
Now observe that
\begin{multline}
J^{NE}_{X,t}(\hat v;\bar \xi,\bar \psi) - J^{LQ}_{x,t}(\hat v) 
\\
= \bb{E}\left\{\int_t^T \left[\langle \bar Q(s)\bar X(s),\bar X(s) \rangle + \langle \bar R(s)\bar v(s),\bar v(s) \rangle + \langle \bar S(s)\bar X(s),\bar v(s) \rangle\right]ds + \langle \bar H \bar X(T),\bar X(T) \rangle \right\}.
\end{multline}
We call this difference the ``price of anarchy," since it is the added aggregate cost of allowing all players to independently choose their optimal strategy.

\subsection{$\epsilon$-Nash equilibrium of an $N$-player game}

In this section we discuss the relationship between the mean field game given in the previous section and the analogous $N$-player game, which can be formulated as follows.
First, we specify that $W_1,\ldots,W_N$ are $N$ independent Wiener processes, and $X_1,\ldots,X_N$ are $N$ i.i.d.~random variables.
The state of player $i \in \{1,2,\ldots,N\}$ is given by the dynamics
\begin{multline} \label{lq N player game dynamics}
dX_i(s) = \left\{A(s)X_i(s) + \bar{A}(s)\frac{1}{N-1}\sum_{j \neq i}X_j(s) + B(s)v_i(s) + \bar{B}(s)\frac{1}{N-1}\sum_{j\neq i}v_j(s) \right\}ds
\\
+ \left\{C(s)X_i(s) + \bar{C}(s)\frac{1}{N-1}\sum_{j \neq i}X_j(s) + D(s)v_i(s) + \bar{D}(s)\frac{1}{N-1}\sum_{j\neq i}v_j(s)\right\}dW_i(s)
\\
+ \left\{F(s)X_i(s) + \bar{F}(s)\frac{1}{N-1}\sum_{j \neq i}X_j(s) + G(s)v_i(s) + \bar{G}(s)\frac{1}{N-1}\sum_{j\neq i}v_j(s)\right\}dW_0(s), \ \ X_i(t) = X_i
\end{multline}
where $v_i$ is the control chosen by player $i$.
The cost functional for player $i$ is given by
\begin{multline} \label{N player game objective functional}
J^{N,i}_{X_i,t}(v_1,\ldots,v_N) = J^{N,i}_{X_i,t}(v_i;\{v_j\}_{j \neq i})
\\ = \bb{E}\biggl\{\int_t^T \bigl[\langle Q(s)X_i(s),X_i(s) \rangle  + 2\frac{1}{N-1}\sum_{j\neq i}\langle \bar Q(s)X_j(s),X_i(s) \rangle  
\\
+ \langle R(s)v_i(s),v_i(s) \rangle + 2\frac{1}{N-1}\sum_{j\neq i}\langle \bar R(s) v_j(s),v_i(s) \rangle
+ 2\langle S(s)X_i(s),v_i(s) \rangle 
\\
+ 2\frac{1}{N-1}\sum_{j\neq i}\langle \bar S_1(s) X_j(s),v_i(s) \rangle + 2\frac{1}{N-1}\sum_{j\neq i}\langle \bar S_2(s) X_i(s), v_j(s) \rangle
\\
+ 2\langle q(s),X_i(s)\rangle + 2\frac{1}{N-1}\sum_{j\neq i}\langle \bar q(s), X_j(s)\rangle
+ 2\langle r(s),v_i(s) \rangle + 2\frac{1}{N-1}\sum_{j\neq i}\langle \bar r(s), v_j(s) \rangle \bigr]ds
\\
+ \langle HX_i(T),X_i(T) \rangle + 2\frac{1}{N-1}\sum_{j\neq i}\langle \bar H X_j(T), X_i(T) \rangle \biggr\}.
\end{multline}
We seek to prove that the mean field equilibrium $\hat v$ can be used as an approximate Nash equilibrium for the $N$-player game, in a way that is stated precisely below.
\begin{definition}
\label{def epsilon equilibrium}
We say that $\{\hat v_i\}_{i=1}^N$ is an ``$\epsilon$-Nash equilibrium" for the $N$-player game provided that, for all $i \in \{1,\ldots,N\}$,
\begin{equation}
J^{N,i}_{X_i,t}(v_i;\{\hat v_j\}_{j \neq i}) \geq J^{N,i}_{X_i,t}(\hat v_i;\{\hat v_j\}_{j \neq i}) - \epsilon
\end{equation}
for any set of controls $\{v_i\}_{i=1}^N$.
\end{definition}
See, for instance, \cite{cardaliaguet2010notes,huang2003individual,huang2007large,huang2006large}.
The following theorem states that the mean field Nash equilibrium is an $\epsilon$-Nash equilibrium for the $N$-player game.
However, we are unable to prove it here in the general case.
Instead, we restrict our attention to the case where the mean field game is equivalent to an optimal control problem, as in Proposition \ref{mfg = mftc}.
\begin{theorem} \label{epsilon equilibrium}
Let $\bar A, \bar B, \bar C, \bar D, \bar F, \bar G, \bar q,$ and $\bar r$ all be zero.
Additionally, assume $\bar S = \bar S_1 = \bar S_2$.

Let $v^*_i$ be a mean field Nash equilibrium for \eqref{mfg objective functional} with $X = X_i$ and $W = W_i$.
Then for any $\epsilon > 0$ there exists $N_\epsilon$ large enough such that if $N \geq N_\epsilon$, then $\{v^*_i\}_{i=1}^N$ is an $\epsilon$-Nash equilibrium for the $N$-player game.
\end{theorem}

\begin{proof}
To begin with, we write down the dynamics which, given the hypotheses of the theorem, are much simpler than \eqref{lq N player game dynamics}:
\begin{equation} \label{simpler lq N player game dynamics}
dX_i = (AX_i + Bv_i)ds
+ (CX_i + Dv_i)dW_i
+ (FX_i + Gv_i)dW_0, \ \ X_i(t) = x_i
\end{equation}
as well as the slightly simplified cost functionals:
\begin{multline} \label{simpler N player game objective functional}
J^{N,i}_{X_i,t}(v_1,\ldots,v_N) = J^{N,i}_{X_i,t}(v_i;\{v_j\}_{j \neq i})
\\ = \bb{E}\biggl\{\int_t^T \bigl[\langle QX_i,X_i \rangle  + 2\frac{1}{N-1}\sum_{j\neq i}\langle \bar QX_j,X_i \rangle  
+ \langle Rv_i,v_i \rangle + 2\frac{1}{N-1}\sum_{j\neq i}\langle \bar R v_j,v_i \rangle
+ 2\langle SX_i,v_i \rangle 
\\
+ 2\frac{1}{N-1}\sum_{j\neq i}\langle \bar S X_j,v_i \rangle + 2\frac{1}{N-1}\sum_{j\neq i}\langle \bar S X_i, v_j \rangle
+ 2\langle q,X_i\rangle
+ 2\langle r,v_i \rangle \bigr]ds
\\
+ \langle HX_i(T),X_i(T) \rangle + 2\frac{1}{N-1}\sum_{j\neq i}\langle \bar H X_j(T), X_i(T) \rangle \biggr\}.
\end{multline}
Now let us establish some notation.
Fix $i \in \{1,\ldots,N\}$.
We will let $v_{i}$ be an arbitrary given control.
Recall that $v^*_j, j \in \{1,\ldots,N\}$ is the mean field Nash equilibrium for \eqref{mfg objective functional} with initial condition $X_j$.
\begin{itemize}
\item Let $X_j^*(s)$, $j \in \{1,\ldots,N\}$ be the corresponding state as given by the mean field dynamics \eqref{lq mfg dynamics} with initial condition $X_j$ and $W = W_j$.
Note that under our assumptions, $X_j^*(s)$ is also the solution to the $N$-player game dynamics given by \eqref{lq N player game dynamics} with $v_j = v^*_j$.
\item Let $X_{i}(s)$ refer to the solution of the system given by \eqref{lq N player game dynamics} with  an arbitrary given control $v_{i}$.
\end{itemize}

Now, as a first step towards the necessary estimates to prove the theorem, we compare the difference between the mean field cost and that of the $N$-player game.
 \begin{multline}
 J_{x_i,t}^{mfg}(v^*_i;\bar X_i^*,\bar v^*_i) -
 J_{x_i,t}^{N,i}(v^*_i; \{v^*_j\}_{j\neq i})
 \\
 = 2\bb{E}\left\{\int_t^T \left[ \left\langle \bar Q  \left(\bar X^*_i - \frac{1}{N-1}\sum_{j \neq i} X_j^*\right),X_i^* \right\rangle 
 +
 \left\langle \bar R \left(\bar v^*_i - \frac{1}{N-1}\sum_{j \neq i} v_j^*\right),v_i^* \right\rangle \right. \right.
 \\
 \left. +  \left\langle \bar S \left(\bar X_i^* - \frac{1}{N-1}\sum_{j \neq i} X_j^*\right),v^*_i \right\rangle   + \left\langle \bar S X_i^*,\bar v^*_i - \frac{1}{N-1}\sum_{j \neq i} v_j^* \right\rangle \right] ds  
 \\
 \left. + \left\langle \bar H \left(\bar X^*_i(T) - \frac{1}{N-1}\sum_{j \neq i} X_j^*(T)\right),X^*_i(T) \right\rangle \right\}
 \end{multline}
Observe that the processes $X_i^*$ are conditionally i.i.d.
Since, in addition, they are continuous in $L^2(\Omega)$, we have, as $N \to \infty$,
\begin{equation}
\bar X^*_i(s) - \frac{1}{N-1}\sum_{j \neq i} X_j^*(s) \to 0 \ \text{in} \ L^2(\Omega), \ \text{uniformly in}~ s \in [t,T].
\end{equation}
Next, we recall that, by Proposition \ref{mfg = mftc} and Theorem \ref{riccati solver}, we have a formula for $v^*_i$ in terms of $X^*_i$, namely \eqref{the optimal control} where $P$ and $\Pi$ are the solutions to the Riccati equations \ref{P equation},\eqref{Pi equation}.
Based on these formulas, we can similarly assert that
\begin{equation}
\bar v^*_i - \frac{1}{N-1}\sum_{j \neq i} v_j^* \to 0 \ \text{in} \ L^2.
\end{equation}
Combining these with the a priori bounds on $X_i^*$ and $v_i^*$ in $L^2$ as well as the $L^\infty$ bounds on the coefficients, we see that, as $N \to \infty$,
\begin{equation*}
J_{x_i,t}^{mfg}(v^*_i;\bar X_i^*,\bar v^*_i) -
 J_{x_i,t}^{N,i}(v^*_i; \{v^*_j\}_{j\neq i}) 
 = o(1).
\end{equation*}
By analogous reasoning, we see that
\begin{equation*}
J_{x_i,t}^{mfg}(v_i;\bar X_i^*,\bar v^*_i) -
 J_{x_i,t}^{N,i}(v_i; \{v^*_j\}_{j\neq i}) 
 = o(1)
\end{equation*}
as well.
Now by definition of mean field Nash equilibrium, $v_i^*$ is the optimal control for $J_{x_i,t}^{mfg}(\cdot;\bar X_i^*,\bar v^*_i)$.
Therefore, we have
\begin{align*}
J_{x_i,t}^{N,i}(v_i; \{v^*_j\}_{j\neq i}) &= J_{x_i,t}^{mfg}(v_i;\bar X_i^*,\bar v^*_i) + o(1)\\
&\geq J_{x_i,t}^{mfg}(v^*_i;\bar X_i^*,\bar v^*_i) + o(1)\\
&= J_{x_i,t}^{N,i}(v^*_i; \{v^*_j\}_{j\neq i}) + o(1),
\end{align*}
which is what we wanted to show.
\end{proof}



\section{Example from economics: production of an exhaustible resource} \label{sec:econ}

In this section we develop and analyze a model of exhaustible resource production, following the work of Gu\'eant, Lasry, and Lions in \cite{gueant2011mean} and of Chan and Sircar in \cite{chan2014bertrand,chan2015fracking}.
As in the general mean field setting, we assume the number of producers of a given resource (oil, for example) is very large.
Consider an arbitrary producer.
Let $v(s)$ represent the quantity produced at time $s$, while $X(s)$ is the producer's current level of reserves.
Following \cite{gueant2011mean}, we assume the dynamics are stochastic, with the noise proportional to the current number of reserves.
In addition, we deal with a \textit{common noise,} modeling uncertainty inherent in nature itself, rather than in the measurements of individual producers.
We have
\begin{equation}
\label{production dynamics}
dX(s) = -v(s)ds + \nu X(s) dW(s) + \nu_0X(s) dW_0(s), \ X(t) = x.
\end{equation}
The goal of each individual producer is the maximization of profit.
We model the market competition as a Nash equilibrium.
Define $k(s)$ to be the price at which a producer can sell, and define $\bar k(s)$ to be the market price.
To simplify the analysis and allow the model to fall under our linear-quadratic framework, we follow \cite{chan2014bertrand} and consider a linear demand schedule
\begin{equation} \label{demand schedule}
v(s) = \gamma + \delta\bar k(s) - k(s)
\end{equation}
for given parameters $\gamma,\delta$.
In \cite{chan2014bertrand} they are given by
\begin{equation}
\label{epsilon parameter}
\gamma = \frac{1}{1+\epsilon}, \ \delta = \frac{\epsilon}{1+\epsilon}
\end{equation}
for a parameter $\epsilon \geq 0$ which measures the degree of competition ($\epsilon = 0$ corresponds to monopoly, as the market price is unseen by consumers, whereas $\epsilon = +\infty$ corresponds to perfect competition, as the market price has exactly the same weight as the price offered by each individual firm).
The revenue maximization problem can now be stated as
\begin{equation}
\label{profit maximization}
\sup_{v} \bb{E}\left[\int_t^T e^{-\mu(s-t)}v(s)k(s)ds - e^{-\mu(T-t)}|X(T)|^2 \right].
\end{equation}
To make this into a linear-quadratic functional of the form \eqref{mfg objective functional}, we must first compute $k(s)$ in terms of $v(s)$ and $\bar \psi(s)$, where we recall that $\bar \psi(s) = \bb{E}[\hat v(s)| \sr{F}_s^0]$ is the conditional expectation of the optimal control.
To do this, it will first be necessary to find a formula for the market price.

Let $\hat k(s)$ be the price corresponding to the optimal quantity $\hat v(s)$.
In equilibrium, the market price is precisely the (conditional) expected value of $\hat k(s)$, so that
$$
\bar k(s) = \bb{E}[\hat k(s)|\sr{F}^0_s] = \bb{E}[\gamma + \delta \bar k(s) - \hat v(s)|\sr{F}^0_s] = \gamma + \delta \bar k(s) - \bar \psi(s)
\ \ 
\Rightarrow
\ \ 
\bar k(s) = \frac{\gamma}{1-\delta} - \frac{1}{1-\delta}\bar \psi(s).
$$
Hence from \eqref{demand schedule} we get
$$
k(s) = \frac{\gamma}{1-\delta} - \frac{\delta}{1-\delta}\bar \psi(s) - v(s).
$$
Then \eqref{profit maximization} becomes
\begin{equation}
\label{profit maximization2}
-\inf_{v} \bb{E}\left[\int_t^T e^{-\mu(s-t)}(v^2(s) + 2\beta v(s)\bar \psi(s) - 2\alpha v(s)) ds + e^{-\mu(T-t)}|X(T)|^2 \right]
\end{equation}
where we have set
\begin{equation}
\alpha := \frac{\gamma}{2(1-\delta)}, \ \beta := \frac{\delta}{2(1-\delta)}.
\end{equation}
The dynamics and objective functional above are a much simplified form of \eqref{lq mfg dynamics}, with
$$
A = \bar A = \bar B = \bar C = D = \bar D = \bar F = G = \bar G = 0,
B = -1, C = \nu, F = \nu_0.
$$
Likewise in the cost functional \eqref{mfg objective functional} we have
\begin{align*}
 & Q = \bar Q = \bar H = S = \bar S = q = \bar q = \bar r = 0,
 \\ 
 & R(s) = e^{-\mu(s-t)}, \bar R(s) = \beta e^{-\mu(s-t)}, r(s) = -\alpha e^{-\mu(s-t)}, H = e^{-\mu(T-t)}.
\end{align*}
Moreover, this setup satisfies the assumptions of Proposition \ref{mfg = mftc}, so the Nash equilibrium $\hat v$ quantity is in fact an optimizer for a mean field type control problem.
The corresponding objective functional is
\begin{equation}
J_{x,t}^{LQ}(v(\cdot)) = \bb{E}\left\{\int_t^T e^{-\mu(s-t)}\left(v^2(s) + \beta \left(\bb{E}[v(s)|\sr{F}^0_s]\right)^2 - 2\alpha v(s) \right) ds + e^{-\mu(T-t)}\bb{E}[|X(T)|^2]\right\},
\end{equation}
and the price of anarchy is given by
$$
J_{x,t}^{mfg}(\hat v(\cdot)) - J_{X,t}^{LQ}(\hat v(\cdot)) = \beta \bb{E}\int_t^T e^{-\mu(s-t)}\left(\bb{E}[v(s)|\sr{F}^0_s]\right)^2 ds.
$$
If we take, as in \cite{chan2014bertrand}, the formula $\delta = \epsilon/(1+\epsilon)$, we get simply $\alpha = 1/2$ and $\beta = \epsilon/2$; so we see that the price of anarchy is directly proportional to the competition coefficient $\epsilon$:
$$
J_{X,t}^{mfg}(\hat v(\cdot)) - J_{X,t}^{LQ}(\hat v(\cdot)) = \frac{\epsilon}{2} \bb{E}\int_t^T e^{-\mu(s-t)}\left(\bb{E}[v(s)|\sr{F}^0_s]\right)^2 ds.
$$

\subsection{Computation of the equilibrium strategy}

By the theory developed in the previous section, we can compute the market equilibrium quite explicitly.
The Riccati equations are
\begin{multline}
\left\{
\begin{array}{l}
\dot p + (\nu^2 + \nu_0^2) p - e^{\mu(s-t)}p^2 = 0, \vspace{.2cm}
\\
p(T) = e^{-\mu(T-t)}
\end{array}\right.
\end{multline}
and
\begin{multline}
\left\{
\begin{array}{l}
\dot \pi  + \nu^2 p + \nu_0^2 \pi-(1+\beta)^{-1}e^{\mu(s-t)}\pi^2  = 0, \vspace{.2cm}
\\
\pi(T) = e^{-\mu(T-t)}.
\end{array}\right.
\end{multline}
We can explicitly solve for $p$.
If $\lambda := \mu - (\nu^2+\nu_0^2) \neq 0$, we get
\begin{equation}
p(s) = \frac{\lambda e^{-\mu (s-t)}}{(\lambda+1)e^{\lambda(T-s)}-1}
\end{equation}
while if $\lambda = 0$ we find
\begin{equation}
p(s) = \frac{e^{-\mu (s-t)}}{1+T-s}.
\end{equation}
Let us also remark that in the special case $\nu = 0$ we can even compute $\pi$ explicitly:
\begin{equation}
\pi(s) = \frac{(1+\beta)\lambda e^{-\mu (s-t)}}{((1+\beta)\lambda+1)e^{\lambda(T-s)}-1} \ \text{if} \ \lambda \neq 0, \ \frac{(1+\beta)e^{-\mu (s-t)}}{1+\beta+T-s} \ \text{if} \ \lambda = 0.
\end{equation}

For the optimal trajectory, we have, by Equation \eqref{X},
\begin{equation}
d\bar X(s) = \left(-\frac{1}{1+\beta}e^{\mu (s-t)}\pi(s)\bar X(s) - \frac{\alpha}{1+\beta}\right)ds + \nu_0 \bar X(s)dW_0(s)
\end{equation}
and
\begin{equation}
d(X(s)-\bar X(s)) =
 -e^{\mu (s-t)}p(s)(X-\bar X(s))ds + \nu X(s)dW(s) + \nu_0(X-\bar X(s))dW_0(s).
\end{equation}
These can be solved explicitly in terms of $\pi$ and $p$.
We have
\begin{equation} \label{econ optimal EX}
\bar X(s) = e^{\Psi(s)}\left(\bb{E}[x|\sr{F}_t^0]- \frac{\alpha}{1+\beta}\int_t^s e^{-\Psi(\tau)} d\tau\right)
\end{equation}
where
\begin{equation}
\Psi(s) := \frac{1}{1+\beta}\int_t^s e^{\mu(\tau-t)}\pi(\tau)d\tau + \frac{\nu_0^2}{2}(s-t) + \nu_0(W_0(s)-W_0(t)),
\end{equation}
which we then use to find that
\begin{equation} \label{econ optimal X-EX}
X(s)-\bar X(s) 
= e^{\Phi(s)} \left(x-\bb{E}[x|\sr{F}_t^0] - \nu^2 \int_t^s e^{-\Phi(\tau)} \bar X(\tau)d\tau + \nu \int_t^s e^{-\Phi(\tau)} \bar X(\tau)dW(\tau)\right)
\end{equation}
where
\begin{equation}
\Phi(s) := \int_t^s e^{\mu(\tau-t)}p(\tau)d\tau + \frac{1}{2}(\nu^2+\nu_0^2)(s-t) + \nu (W(s)-W(t))+\nu_0 (W_0(s)-W_0(t)).
\end{equation}
Then we have the following for the optimal control, by Equation \eqref{v}:
\begin{equation}
\label{econ optimal q}
v(s) = e^{\mu (s-t)}\left(p(s)(X(s)-\bar X(s)) + \frac{1}{1+\beta}\pi(s)\bar X(s) + \frac{\alpha}{(1+\beta)^2}\int_t^s \pi(\tau)d\tau \right) + \frac{\alpha}{1+\beta}.
\end{equation}

\subsection{Market price}

From Equation \eqref{econ optimal q} and formula \eqref{demand schedule}, we see that the market price is given by
\begin{equation} \label{market price}
\bar k(s) = 2\alpha - (1+2\beta)\bar v(s) 
= \frac{\alpha}{1+\beta} - e^{\mu(s-t)}\left(\frac{1+2\beta}{1+\beta}\pi(s)\bar X(s) + \frac{\alpha(1+2\beta)}{(1+\beta)^2}\int_t^s \pi(\tau)d\tau \right)
\end{equation}
which, when using the model of Chan and Sircar \cite{chan2014bertrand} so that $\alpha = 1/2$ and $\beta = \epsilon/2$, becomes
\begin{equation} \label{market price epsilon}
\bar k(s) 
= \frac{1}{2+\epsilon} - \frac{2+2\epsilon}{2+\epsilon}e^{\mu(s-t)}\left(\pi(s)\bar X(s) + \int_t^s \pi(\tau)d\tau \right).
\end{equation}
An interesting question is the behavior of the market price as the competition parameter $\epsilon$ increases.
Let us focus on the expected market price:
\begin{equation}
\label{expected market price}
\bb{E}[\bar k(s)]
= \frac{1}{2+\epsilon} - \frac{2+2\epsilon}{2+\epsilon}e^{\mu(s-t)}\left(\pi(s)\chi(s) + \int_t^s \pi(\tau)d\tau \right).
\end{equation}
where we define $\chi(s) = \bb{E}[X(s)] = \bb{E}[\bar X(s)]$.
It is possible to give conditions on the initial data such that the expected value of the state variable remains positive up to time $T$.
To see this, note that
$$
\chi'(s) = -\frac{1}{1+\beta}e^{\mu(s-t)}\pi(s)\chi(s)  - \frac{\alpha}{1+\beta}, \ \ \chi(t) = \bb{E}[x].
$$
We solve to get
$$
\chi(s) = \left(\bb{E}[x] - \frac{\alpha}{1+\beta}\int_t^s \exp\left\{\frac{1}{1+\beta}\int_t^\sigma \tilde \pi(\tau)d\tau\right\}d\sigma\right)\exp\left\{-\frac{1}{1+\beta}\int_t^s \tilde \pi(\tau)d\tau\right\}
$$
where $\tilde \pi(s) = e^{\mu(s-t)}\pi(s)$.
By analyzing $\tilde \pi$, we can find conditions under which $\chi(s)$ will be positive.
Note that
$$
\tilde \pi' + \nu^2 \tilde p + (\nu_0^2-\mu)\tilde \pi - \frac{1}{1+\beta}\tilde \pi^2 = 0, \ \ \tilde \pi(T) = 1
$$
where $\tilde p(s) = e^{\mu(s-t)}p(s)$.
We use the inequality $ab \leq \frac{1}{4}a^2 + b^2$ and the fact that $\tilde p \leq 1$ to get
$$
-\tilde \pi' \leq \nu^2 + \frac{1}{4}(1+\beta)(\nu_0^2-\mu)^2
$$
which yields
$$
\tilde \pi(s) \leq 1 + \frac{1}{4}\left((1+\beta)(\nu_0^2-\mu)^2 +4\nu^2\right)(T-s) \leq \kappa := 1 + \frac{1}{4}\left((1+\beta)(\nu_0^2-\mu)^2 +4\nu^2\right)T.
$$
Using this estimate we deduce
$$
\chi(s) \geq \left(\bb{E}[x] - \frac{\alpha}{\kappa}\left(e^{\kappa(s-t)/(1+\beta)}-1\right)\right)\exp\left\{-\frac{1}{1+\beta}\int_t^s \tilde \pi(\tau)d\tau\right\}
$$
Therefore if we have
\begin{equation} \label{small T}
\frac{\alpha}{\kappa}\left(e^{\kappa(T-t)/(1+\beta)}-1\right) < \bb{E}[x]
\ \ \Leftrightarrow \ \
T-t < \frac{1+\beta}{\kappa}\ln \left(1 + \frac{\kappa}{\alpha}\bb{E}[x]\right),
\end{equation}
that is, if $T-t$ is small enough, we have $\chi(s) = \bb{E}[X(s)] \geq 0$ for all $s \in [t,T]$.
We can interpret this smallness condition as saying that, on average, the initial reserves are not used up by time $T$.
Recalling that $\beta = \epsilon/2$ and $\alpha = 1/2$, we notice that condition \eqref{small T} is equivalent to
$$
T-t < \frac{4+2\epsilon}{4 + \left((1+\epsilon/2)(\nu_0^2-\mu)^2 +4\nu^2\right)T}\ln \left(1 + \frac{1}{2}\left(4 + (1+\epsilon/2)(\nu_0^2-\mu)^2 +4\nu^2\right)T\bb{E}[x]\right)
$$
where the right-hand side is large when $\epsilon$ is large.

Additionally, we observe that
$$
-\tilde \pi' \geq (\nu_0^2-\mu)\tilde \pi - \frac{1}{1+\beta}\tilde \pi^2
$$
using the fact that $\tilde p \geq 0$.
Using the substitution $u = \tilde \pi^{-1}$ we have
$$
u' \geq (\nu_0^2-\mu)u - \frac{1}{1+\beta} \ \ \Rightarrow \ \ u(s) \leq \left(1+\frac{T-s}{1+\beta}\right)e^{|\nu_0^2-\mu|(T-s)}
\ \
\Rightarrow \ \
\tilde \pi(s) > 0 \ \forall \ s \in [t,T].
$$
It follows that $\pi$ is positive.
Therefore, under condition \eqref{small T} (in particular, for $\epsilon$ large enough with respect to $T$) we have that the expected market price
$$
\bb{E}[\bar k(s)] 
= \frac{1}{2+\epsilon} - \frac{2+2\epsilon}{2+\epsilon}e^{\mu(s-t)}\left(\pi(s)\chi(s) + \int_t^s \pi(\tau)d\tau \right)
$$
is decreasing in $\epsilon$ and goes to zero as $\epsilon \to \infty$.

\section{Conclusion}

In this paper we have discussed the solution of a linear-quadratic mean field type control problem with a common noise and a dependence on the conditional expectation of both state and control variables.
We then compared this to mean field games, where it is seen that in certain cases, the two problems are the same, with a difference in the objective functionals which is called the \textit{price of anarchy}.
We then applied this to an economic model of production of exhaustible resources.
Since it is natural for such aggregate quantities as the expected value of the control to appear in economic models, it is useful to note that variational methods can be used to study the Nash equilibrium in this case.
It would be interesting to pursue this approach in future work on more general models than the linear-quadratic setting.

\bibliographystyle{siam}
\bibliography{C:/mybib/mybib}
\end{document}